\documentclass{siamonline220329}
\usepackage{amsfonts}
\usepackage{graphicx}

\usepackage{amssymb,amsmath}

\title{
A unified approach to reverse engineering and data selection for unique network identification 

\thanks{Submitted to the editors DATE.
\funding{AV-C  was  partially  supported  by  the  Simons  Foundation  (grant  516088). ESD and VN-S were partially supported by the Research, Scholarly \& Creative Activities Program awarded by the Cal Poly Division of Research, Economic Development \& Graduate Education.}}
}

\author{Alan Veliz-Cuba\thanks{Mathematics Department, University of Dayton, Dayton, OH (\email{avelizcuba1@udayton.edu}).}
\and Vanessa Newsome-Slade\thanks{Mathematics Department, California Polytechnic State University (Cal Poly), San Luis Obispo, CA (\email{vnewsome@calpoly.edu}, \email{edimitro@calpoly.edu}).}
\and Elena S. Dimitrova\footnotemark[3]}

\newsiamremark{remark}{Remark}
\newsiamremark{hypothesis}{Hypothesis}
\crefname{hypothesis}{Hypothesis}{Hypotheses}
\newsiamthm{claim}{Claim}
\newsiamremark{example}{Example}

\def\F{\mathbb{F}}
\def\M{\mathcal{M}}
\def\D{\mathcal{D}}

\newcommand{\<}{\langle}
\renewcommand{\>}{\rangle}

\usepackage{xcolor}
\newcounter{comments}
\definecolor{Red}{rgb}{0.8,0,0}
\definecolor{Green}{rgb}{0,0.7,0}
\definecolor{Blue}{rgb}{.2,.2,1}

\begin{document}

\maketitle

\begin{abstract}
Due to cost concerns, it is optimal to gain insight into the connectivity of biological and other networks using as few experiments as possible. Data selection for unique network connectivity  identification has been an open problem since the introduction of algebraic methods for reverse engineering for almost two decades. In this manuscript we determine what data sets uniquely identify the unsigned wiring diagram corresponding to a system that is discrete in time and space. Furthermore, we answer the question of uniqueness for signed wiring diagrams for Boolean networks. 
Computationally, unsigned and signed wiring diagrams have been studied separately, and in this manuscript we also show that there exists an ideal capable of encoding both unsigned and signed information. This provides a unified approach to studying reverse engineering that also gives significant computational benefits.

\end{abstract}

\begin{keywords}
Wiring diagram, squarefree monomial ideal, abstract simplicial complex, reverse engineering, data selection, polynomial dynamical system
\end{keywords}

\begin{MSCcodes}
37N25, 94C10, 70G55
\end{MSCcodes}

\section{Introduction}

Biological systems are commonly represented using networks which provide information about interactions between various elements in the system. Knowledge of the network connectivity is crucial for studying network robustness, regulation, and control strategies in order to develop, for example, therapeutic interventions~\cite{tan2013, Wang:2013aa} and drug delivery strategies~\cite{yousefi2012,Lee:2012aa}, or to understand the mechanisms for the spread of an infectious disease~\cite{Madrahimov:2013dq,PMID:20478257}. Moreover, it has been demonstrated that the role of network connectivity goes beyond static properties and can in fact dictate certain dynamical properties and be used for their control~\cite{jarrah2010dynamics,campbell,veliz2011reduction,zamal,wu,albert,sontag2008effect,murrugarra15, murrugarra19}. 

Reverse engineering is one approach by which network connectivity can be reconstructed by viewing the system in question as a black box and only considering the available experimental data in order to gain insight into the system's inner connections. These connections are encoded using \emph{wiring diagrams}, which are directed graphs that describe the relationships between elements and how elements in a network affect one another. A \emph{signed wiring diagram} provides additional information about the interactions between elements; for example, in the context of a gene regulatory network, a signed wiring diagram reflects whether a gene acts as an activator or an inhibitor to another gene.

Due to cost concerns related to conducting numerous experiments, it is of interest to determine the simplest possible model that fits the data. 
 One approach consists of finding all functions that fit the data and selecting the simplest under some criteria \cite{laubenbacher04}. However, this can result in several possible candidates, especially if the data set is small. Another approach is to use wiring diagrams, where each of them can be seen as an equivalent class of all functions that depend on a given set of variables \cite{Jarrah_reveng,Veliz_reveng}. 
 In this case, we want to find the smallest sets of variables required to fit the data. To accomplish this, we consider \emph{minimal} wiring diagrams. Given a set of inputs and outputs, algorithms in \cite{Jarrah_reveng} and \cite{Veliz_reveng} compute all possible minimal wiring diagrams using the primary decomposition of ideals. Furthermore, we wish to determine data sets that \emph{uniquely} determines the smallest set of variables required to fit the data, which is the focus of this manuscript.

\subsection{Motivation}
Stanley-Reisner theory provides a bijective correspondence between squarefree monomial ideals and abstract simplicial complexes. Since the algorithm introduced in~\cite{Jarrah_reveng} for compute unsigned minimal wiring diagrams relies on squarefree monomial ideals, we are able to utilize this correspondence in order to determine which sets of inputs always correspond to a unique unsigned minimal wiring diagram regardless of the output assignment. However, there is not an established connection between abstract simplicial complexes and signed minimal wiring diagrams. Such a connection would be beneficial for interpreting these ideals as squarefree monomial ideals that retain the information about the signs of interactions so that the correspondence provided by Stanley-Reisner theory can still be used.

We use the information gleaned from the unsigned case to determine the conditions under which input sets have a unique signed minimal wiring diagram regardless of output assignment. Having a unique minimal wiring diagram provides us with one set of variables required to be consistent with the given data. In reverse engineering a biological system, it is desirable to find the network with the least number of experiments possible. Knowing that a set of inputs will provide one minimal wiring diagram without knowing the outcome of the experiment given by a set of outputs is beneficial in gaining information about a biological network, while minimizing the number of experiments to be conducted.

Theorem \ref{thm:unified} provides an approach to study all min-sets (unsigned and signed) with a single mathematical object. 
Theorem \ref{thm:uniqueness} describes sufficient conditions of uniqueness from the structure of the possible set of generators of the ideal that encodes the data.
Theorems \ref{thm:uniqueness_unsigned_Boolean} and \ref{thm:uniqueness_signed_Boolean} give necessary and sufficient conditions for uniqueness of wiring diagrams of Boolean networks using the structure of the inputs in the hypercube. Furthermore, Theorem~\ref{thm:uniqueness_unsigned_Boolean} is valid for any number of states.

\section{Background}

\subsection{Unsigned and signed wiring diagrams}
\begin{definition}[Polynomial dynamical system]
	A \textit{polynomial dynamical system} (PDS) over a finite field $\F$ is a function $$f=(f_1,\ldots,f_n):\F^n \rightarrow \F^n$$ with coordinate functions $f_i \in \F[x_1, \ldots, x_n]$.
\end{definition}
Throughout this discussion, we are reverse-engineering polynomial dynamical systems node by node, focusing on one component function at a time. 

We are interested in functions that represent biological regulation which are functions whose variables only affect the function either positively or negatively, in the sense that increasing the value of $x_i$ increases (respectively, decreases) the output of $f$. Functions with these properties are called \emph{unate} functions (also known as \emph{monotone} functions), formally defined  as $f:\F^n\to \F$ such that for all $i=1,\ldots, n$, $f$ does no depend on $x_i$, $f$ depends positively on $x_i$, or $f$ depends negatively on $x_i$.

It is standard to define the \emph{support} of any function $f$, supp$(f)$, as the collection of variables that appear in $f$. In the context of unate function we will distinguish between variables that affect the function ``positively'' and those that affect it ``negatively,'' motivating the following definitions. 

\begin{definition}[Support of a unate function] The \emph{positive support} of a unate function $f$ is defined as 
$
\textrm{supp}^+(f)=\{x_i | x_i \textit{~is~an~activator~of~} f \}
$ and 
the \emph{negative support} of a unate function $f$ is 
$
\textrm{supp}^-(f)=\{x_i | x_i \textit{~is~an~inhibitor~of~} f \}.
$
To encode the support of a unate function as a single set while still keeping the information of the signs of the variables, we use $\textrm{supp}^{sgn}(f)=\{x_i|x_i\in \textrm{supp}^+(f)\}\cup \{\overline{x_i}|x_i\in \textrm{supp}^-(f)\}$ which we call the \emph{signed} support of $f$ ($\overline{ x_i }$ is used to denote the fact that $x_i$ is an inhibitor).
\end{definition}

Note that the support and signed support of a constant function is the empty set.

\begin{example}
    Consider $f:\{0,1\}^3\rightarrow \{0,1\}$ given by $f(x_1,x_2,x_3)=x_1 \wedge \overline{x_3}$ and $g:\{0,1,2\}^4\rightarrow \{0,1,2\}$ given by $g(x_1,x_2,x_3,x_4)=\max(\min(x_1,2-x_2),x_4)$. The signed supports are  $supp^{sgn}(f)=\{x_1,\overline{x_3}\}$ and $supp^{sgn}(g)=\{x_1,\overline{x_2},x_4\}$.  
\end{example}

A \emph{wiring diagram} of a PDS $F=(f_1,\ldots,f_n)$ is a directed graph, where the vertices are labeled as the $n$ variables, and there is a directed edge $x_i\rightarrow x_j$ iff $x_i\in supp(f_j)$. On the other hand, the
\emph{signed wiring diagram} is a directed graph, where edges can stand for activation or inhibition (but not both). We draw an arrow from $x_i$ to $x_j$ if and only if $x_i\in \textrm{supp}^+(f_j)$ and we draw a blunt edge from $x_i$ to $x_j$ if and only if $x_i\in \textrm{supp}^-(f_j)$. 

\begin{example}\label{eg:BN}
    Consider $F:\{0,1\}^3\rightarrow \{0,1\}^3$ given by $F(x)=(x_2, x_1 \wedge \overline{x_3}, x_1 \vee x_3)$. Its wiring diagram is shown in Fig.~\ref{fig:wd}a. Since knowing the supports of each function in a network is sufficient to reconstruct the wiring diagram (Fig.~\ref{fig:wd}b), we will focus on reverse-engineering the ``local'' wiring diagram of functions from partial information. Note that a local wiring diagram can be seen as a graphical representation of the support of a function. 
\end{example}

\begin{figure}[ht]\label{fig:wd}
\centering
    \includegraphics[width=0.8\textwidth]{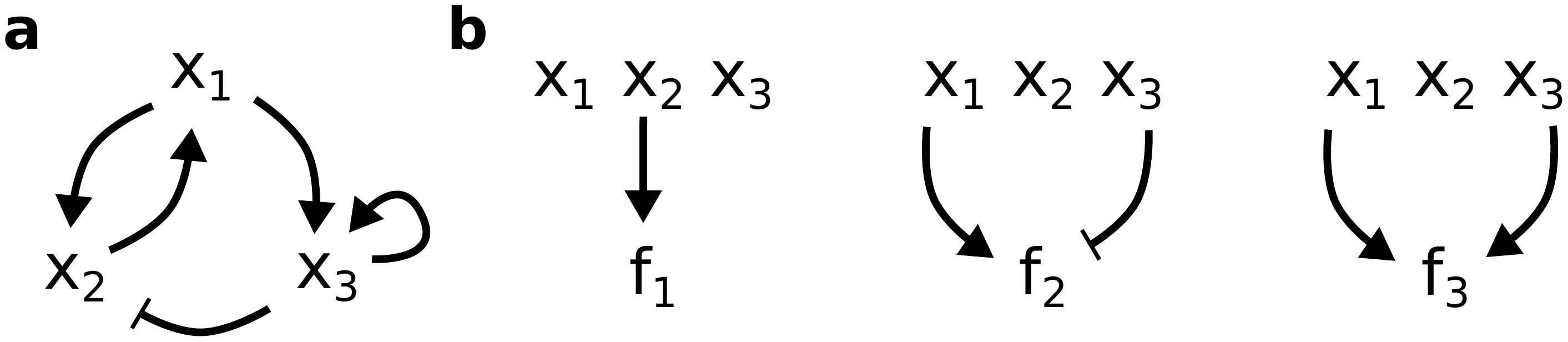}
    \caption{Wiring diagram. (a) Wiring diagram for Example \ref{eg:BN}. (b) ``Local'' wiring diagrams for each of $f_1$, $f_2$, $f_3$. These wiring diagrams are graphical representations of the supports. Note that these three wiring diagrams are sufficient to reconstruct the wiring diagram in (a). Also, these local wiring diagrams are graphical representations of the supports $supp^{sgn}(f_1)=\{x_2\}$, $supp^{sgn}(f_2)=\{ x_1,\overline{x_3}\}$, $supp^{sgn}(f_2)=\{ x_1,x_3\}$.
}
\end{figure}
Given partial information about a function, we form the (input-output) data set $\D = \{(s_1, t_1), \ldots, (s_m,t_m)\}$ where each $s_i \in \mathbb{F}^n$ and $t_i \in \mathbb{F}$.

\begin{definition}[Model space]
    The \emph{model space} of $\D$, denoted Mod$(\D)$, is the set of functions that fit the data; that is, $$\textrm{Mod}(\D) = \{f: \mathbb{F}^n \rightarrow \mathbb{F} ~|~ f(s_i)=t_i \textrm{ for all } i=1,\ldots,m\}.$$
	The \emph{signed model space} of $\D$, denoted Mod$^{sgn}(\D)$, is the set of unate functions that fit the data, i.e.
 $$ \textrm{Mod}^{sgn}(\D) = \{f: \mathbb{F}^n \rightarrow \mathbb{F} ~|~ f \textrm{~is~unate~and~} f(s_i)=t_i \textrm{ for all } i=1,\ldots,m\}.$$ 
\end{definition}

Note that for an arbitrary data set $\mathcal{D}$, the signed model space $ \textrm{Mod}^{sgn}(\D)$ may be empty. However, if we know that the data comes from a unate function, then $ \textrm{Mod}^{sgn}(\D)$ is guaranteed to have at least one element.

\begin{example}\label{ex1}
Suppose we have a data set for a Boolean function $f$ with input of the form $(x_1,x_2,x_3)$ such that
\begin{table}[H]
\centering
    \begin{tabular}{|c||c|c|c|}
    \hline
    input & $(1,1,1)$ & $(0,0,0)$ & $(1,1,0)$ \\
    \hline
    output & 0 & 0 & 1\\
    \hline
\end{tabular}
\end{table}

Since five outputs are unknown, there are $2^5$ Boolean functions that fit the data. That is, $Mod( \mathcal{D})$ has 32 elements.  
On the other hand, there is no general formula for the size of the signed model space. By exhaustive enumeration, it can be shown that $Mod^{sgn}( \mathcal{D})$ has four elements. Namely, 
$Mod^{sgn}( \mathcal{D})=\{
x_1 \wedge \overline{x_3},
x_2 \wedge \overline{x_3},
x_1 \wedge x_2 \wedge \overline{x_3},
(x_1 \vee x_2) \wedge \overline{x_3}
\}$. The wiring diagrams of the functions in $Mod^{sgn}( \mathcal{D})$ are shown in  Fig.~\ref{wds} (the last one is the wiring diagram of two unate functions). Note that (a) and (b) are minimal elements of $Mod^{sgn}( \mathcal{D})$ (with respect to inclusion), while (c) is not. These are graphical representations of the supports of the unate functions in $Mod^{sgn}( \mathcal{D})$: $\{x_1,\overline{x_3}\}$, $\{x_2,\overline{x_3}\}$, and $\{x_1,x_2,\overline{x_3}\}$.

\begin{figure}[ht]\label{wds}
\centering
\includegraphics[width=0.5\textwidth]{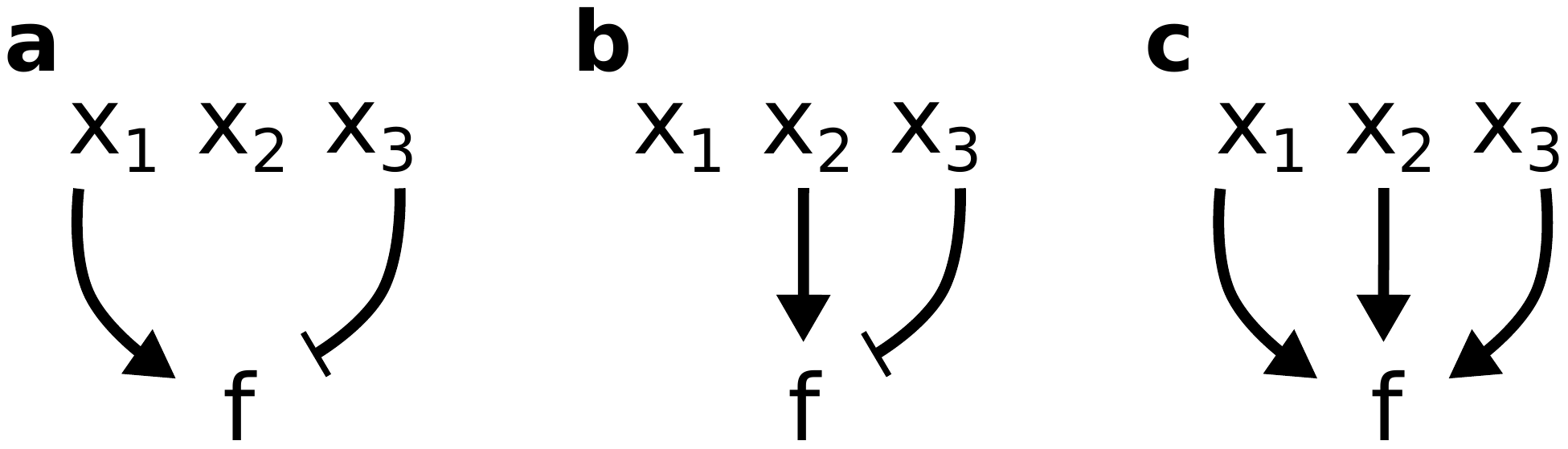}
    \caption{Possible (local) wiring diagrams for the function $f$ in Example \ref{ex1}. 
    }
\end{figure}

The unsigned case provides information about which variables affect $x_j$; activators and inhibitors are not specified. In this example the unsigned sets are $\{x_1, x_3\},$ $\{x_2, x_3\}$, and $\{x_1, x_2, x_3\}.$ Then, the minimal wiring diagrams will be $\{x_1, x_3\}$ and $\{x_2, x_3\}$.

\end{example}

We will later see that the relationship between unsigned and signed minimal wiring diagrams is not simple as the previous example may suggest. It is not always true that $Mod( \mathcal{D})$ is simply $Mod^{sgn}( \mathcal{D})$ after dropping the negations (Examples \ref{eg:unsigned2} and \ref{eg:signed2}). A data set may have more unsigned than signed minimal wiring diagrams or vice versa (Examples \ref{eg:nonprime_ext_nonprime_unsgn_prime_sgn} and \ref{eg:nonprime_ext_prime_unsgn_nonprime_sgn}). We also note that while a data set will have at least one unsigned wiring diagram, it may not have any signed wiring diagrams if the data was not generated by a unate function (Example~\ref{eg:no_signed_wd}).

The main problem in reverse-engineering the wiring diagrams from data is that the model spaces can have a large number of data points even for small problems. For example, consider the data set $\mathcal{D}$ given in Table~\ref{tab:nonBoolean}. Note that each $s_i \in \mathbb{F}_4^3$ and each $t_i \in \mathbb{F}_4.$ 
\begin{table}[H]
\label{tab:nonBoolean}
    \centering
    \begin{tabular}{|c||c|c|c|c|}
    \hline
    $s_i$ & $(0,2,1)$ & $(1,0,3)$ & $(3,0,3)$ & $(2,3,0)$ \\
    \hline
    $t_i$ & 0 & 0 & 2 & 3\\
    \hline
    \end{tabular}
    \caption{Non Boolean data set.}
\end{table}

In this case there are $4^3-4=60$ inputs with unknown values and hence there are $4^{60}\approx 10^{36}$ functions that fit the data. This is too large to analyze by exhaustive search, but we will use algebraic tools to study the wiring diagrams of functions in $Mod( \mathcal{D})$ and $Mod^{sgn}( \mathcal{D})$ without having to list the functions. 

\subsection{Min-sets}

\begin{definition}[Disposable set]
   A set of variables $Y$ is \emph{disposable} if there exists a function $f \in$ Mod$(\mathcal{D})$ such that supp$(f) \cap Y = \emptyset.$
    
	A set of activators and inhibitors $Y$ is \emph{disposable} if there exists a unate function $f \in$ Mod$^{sgn}(\mathcal{D})$ such that $supp^{sgn}(f) \cap Y = \emptyset.$
\end{definition}

\begin{definition}[Feasible set]
   A set of variables $Y$ is \textit{feasible} if there exists a function $f \in$ Mod$(\mathcal{D})$ such that supp$(f)\subset Y.$
    
	A set of activators and inhibitors $Y$ is \textit{feasible} if there exists a function $f \in$ Mod$^{sgn}(\mathcal{D})$ such that $supp^{sgn}(f) \subset Y.$

\end{definition}

The following definition applies to both signed and unsigned model spaces. 
\begin{definition}[Min-set]

A set of  variables is a \emph{min-set} of $\mathcal{D}$ if it is a minimal feasible set for $\mathcal{D}$.
\end{definition}

Alternatively, we can define a min-set of $\mathcal{D}$ as a set of variables whose complement is a maximal disposable set of $\mathcal{D}$.

The main result in \cite{Jarrah_reveng,Veliz_reveng} is that it is possible to do calculations regarding wiring diagrams of elements in $Mod( \mathcal{D})$ and $Mod^{sgn}( \mathcal{D})$ without having to do any calculations with these sets. More precisely, it is possible to compute the minimal wiring diagrams of functions that have minimal support in $Mod( \mathcal{D})$ and $Mod^{sgn}( \mathcal{D})$ (that is, the min-sets) without listing the actual functions. We now summarize those results.

In \cite{Jarrah_reveng}, the authors developed an algorithm for constructing all unsigned minimal
wiring diagrams based on sets of input-output data $\D$.  The method encodes coordinate changes in the input data $V$ as square-free monomials, generates a monomial ideal from these monomials, and uses Stanley-Reisner theory to  decompose the ideal into primary components that coincide with the min-sets defined above. In summary, for every pair of distinct input vectors $\mathbf{s}=(s_1,\dots,s_n)$ and $\mathbf{s'}=(s'_1,\dots,s'_n)$ in an input data set $V$, we can encode the coordinates in which they differ by a square-free monomial 
$
m(\mathbf{s},\mathbf{s'})=\prod_{s_i\neq s'_i}x_i.
$
The \emph{ideal of non-disposable sets}
\begin{equation}\label{eqn:unsigned_I}
I=\<m(\mathbf{s},\mathbf{s'})\mid t\neq t'\>
\end{equation}
has primary decomposition whose primary components correspond to the the min-sets of $\D$ (proven in \cite{Jarrah_reveng}).

\begin{example} \label{eg:unsigned2}
Consider the data set given in Table~\ref{tab:nonBoolean}. 
We construct the monomials 
$$m(s_1, s_3) = x_1 x_2 x_3 $$
$$m(s_1, s_4)=x_1 x_2 x_3$$
$$m(s_2, s_3)=x_1$$
$$m(s_2, s_4)=x_1 x_2 x_3$$
$$m(s_3, s_4)=x_1 x_2 x_3$$ 

Then, $I = \langle x_1 x_2 x_3, x_1 \rangle$ which has primary decomposition $ I = \langle x_1 \rangle.$ Thus, there is only one (unsigned) min-sets,  $\{x_1\}.$

\end{example}

We now present the algorithm to find all signed min-sets developed in \cite{Veliz_reveng}.\\
\textit{Input:} A set of data $\mathcal{D} = \{(s_1,t_1),\ldots,(s_m,t_m)\}$ where each $s_i \in \mathbb{F}^n$ and $t_i \in \mathbb{F}$. Define $\textrm{sign}(z)$ to be $1$ if $z>0$, $-1$ if $z < 0$, and $0$ if $z=0$.\\
\textit{Output:}  The primary components of $I^{sgn}$ give the signed min-sets. \par
		\textit{Step 1:} Order the data so that the outputs $t_1,\ldots, t_m$ are in non-decreasing order (This step is not needed for the method to work, but makes calculations more efficient).\par
		\textit{Step 2:} For each pair $(s_i,t_i),(s_j, t_j)$ of data points in $\mathcal{D},$ if $t_i \neq t_j$, define a \emph{pseudomonomial} as follows: $$m^{sgn}(s_i,s_j)=\prod_{s_{i k} \neq s_{j k}} (x_k - \textrm{sign}(s_{j k}-s_{i k})).$$\par
		\textit{Step 3:} Let $I^{sgn}=\langle m^{sgn}(s_i, s_j) | t_i < t_j \rangle.$ Note that $I^{sgn}$ is not a monomial ideal. \par
		\textit{Step 4:} Compute the primary decomposition of $I^{sgn}$.
		
The primary components of $I^{sgn}$ correspond to the signed min-sets of $\mathcal{D}$ (proven in \cite{Veliz_reveng}).

\begin{example} \label{eg:signed2}
Consider the data set given in Table~\ref{tab:nonBoolean}.

We construct the pseudomonomials $$m^{sgn}(s_1, s_3) = (x_1-1)(x_2+1)(x_3-1)$$
$$m^{sgn}(s_1, s_4)=m^{sgn}(s_2, s_4)=(x_1-1)(x_2-1)(x_3+1)$$
$$m^{sgn}(s_3, s_4)=(x_1+1)(x_2-1)(x_3+1)$$ 
$$m^{sgn}(s_2, s_3)=(x_1-1)$$
Then, $I^{sgn} = \langle (x_1-1)(x_2-1)(x_3+1), (x_1+1)(x_2-1)(x_3+1), (x_1-1) \rangle$ which has primary decomposition $ I^{sgn} = \langle x_1-1,x_3+1 \rangle \cap \langle x_1-1, x_2-1 \rangle.$ Thus, the signed min-sets are $\{x_1, \overline{x_3}\}$ and $\{x_1, x_2\}.$

\end{example}

The next theorem summarizes several connections between abstract simplicial complexes and squarefree monomial ideals that we will leverage later.

\begin{theorem}\label{tfae}
Let $I$ be a squarefree monomial ideal of $\F[x_1,\ldots,x_n]$, let $\Delta_I$ be the Stanley-Reisner complex of $I$, and let $X=\{x_1,\ldots,x_n\}$. Then the following are equivalent:

\begin{enumerate}
    \item[(a)] $\Delta_I$ has only one facet.
    \item[(b)] The minimal nonfaces of $\Delta_I$ are dimension 0.
    \item[(c)] $I$ is generated by a subset of $X$.
    \item[(d)] $I$ is a prime ideal.
    \item[(e)] Let $M$ be a set of monomials that generate $I$. For each multivariate monomial $m\in M$, there exists a univariate  monomial in $M$ that divides $m$.
\end{enumerate}
\end{theorem}

Part (e) of Theorem~\ref{tfae} is the foundation of Theorem~\ref{thm:uniqueness} which will enable us to algorithmically determine if an input data set will always have a unique min-set regardless of the output.

In the next sections we address the following questions. Is there a way to encode all min-sets regardless of whether they are signed or unsigned? What are the necessary and/or sufficient conditions for uniqueness of min-sets?

\section{A unified approach to signed and unsigned min-sets}

\subsection{Unsigned min-sets and Stanley-Reisner theory}

For unsigned min-sets, in \cite{Jarrah_reveng}, it was established that the set of disposable sets is closed under intersection and, in particular, the set of disposable sets of $\D$ forms an abstract simplicial complex $\Delta_{\mathcal{D}}$. Then, using Stanley-Reisner theory, a bijective correspondence was drawn between $\Delta_{\mathcal{D}}$ and the squarefree monomial ideal $I$ of non-disposable sets defined in (\ref{eqn:unsigned_I}). Furthermore, finding the primary decomposition of $I$ is straightforward: 
For an abstract simplicial complex $\Delta_{\D}$, the primary decomposition of its Stanley-Reisner ideal in $\mathbb{F}[x_1, \dots, x_n]$ is $$I = \bigcap_{\alpha \in \Delta} p^{\overline{\alpha}} = \bigcap_{\substack{\alpha \in \Delta \\ \textrm{maximal}}} p^{\overline{\alpha}}$$  and the primary components of $I$ are the complements of the maximal faces of $\Delta_{\D}$.

\begin{example}
	The simplicial complex corresponding to the ideal $I= \<xy, yz\>$ is $\Delta_I=\{\emptyset, x, y, z, xz\}$. Using this information, we can find the primary decomposition of $I$. The maximal faces of $\Delta_I$ are $y$ and $x z$. So, 
$$
		I = \<xy, yz\>
		= \bigcap_{\substack{\alpha \in \Delta \\ \textrm{maximal}}} p^{\overline{\alpha}} 
		= p^{\overline{y}} \cap p^{\overline{xz}}
		= p^{xz} \cap p^{y}
		= \<x,z\> \cap \<y\>.
$$
	Thus, the primary decomposition of $I$ is $\< x, z\> \cap \<y\>$. 
	
\end{example}

\subsection{Signed min-sets and Stanley-Reisner theory}
In order to gain information about the signs of interactions between variables, the algorithm for signed minimal wiring diagrams uses the primary decomposition of ideals that are not generated by monomials. Due to the ease of computing primary the decomposition of squarefree monomial ideals, 
as well as the desire to treat pseudomonomials ideals as Stanley-Reisner ideals and be able to apply the results in Theorem~\ref{tfae}, we will convert them into squarefree monomial ideals that retain the information about the signs of interactions. This will be done at the expense of doubling the number of variables as explained below; however, even with the increase of variables, the computational benefits are significant.

Consider data set $\mathcal{D} = \{(s_1,t_1),(s_m,t_m),\ldots\}$ where each $s_i \in \mathbb{F}^n$ and $t_i \in \mathbb{F}$. If $t_i\neq t_j$, define the following monomial in the ring with indeterminates in $\{x_1,\ldots,x_n,\overline{x_1},\ldots,\overline{x_n}\}$ (the reason for using $x_i$ and $\overline{x_i}$ is purely mnemonic; one could use $y_i$ and $z_i$ instead, respectively).

\[
m^{ext}(s_i,s_j)  =\prod_{s_{ik}< s_{jk}}x_k \prod_{s_{ik}> s_{jk}}\overline{x_k} 
\]

Then, we define the square-free monomial ideal $I^{ext}=\langle m^{ext}(s_i, s_j) | t_i < t_j \rangle$. 

Note that we can obtain $I$ from $I^{ext}$ by replacing $\overline{x}_j$'s with $x_j$, and $I^{sgn}$ from $I^{ext}$ by replacing $x_j$'s with $x_j-1$ and $\overline{x}_j$'s with $x_j+1$. Because of this the next lemma follows.

\begin{lemma}
Suppose that the primary decomposition of $I^{ext}$ is $I^{ext}=P^{ext}_1\cap P^{ext}_2\cap \ldots \cap P^{ext}_l$. Then, $I=P_1\cap P_2\cap \ldots \cap P_l$, where $P_i$ is the primary ideal obtained from $P^{ext}_i$ by replacing $\overline{x}_j$'s with $x_j$.
Also, $I^{sgn}=P^*_1\cap P^*_2\cap \ldots \cap P^*_l$, where $P^*_i$ is the primary ideal (or the whole ring) obtained from $P^{ext}_i$ by replacing $x_j$'s with $x_j-1$ and $\overline{x}_j$'s with $x_j+1$. 

\end{lemma}

Note that by doing the replacements some ideals may be redundant from the intersection and some ideals may have redundant generators. However, that redundancy can easily be identified and simplified, and thus we will obtain the primary decomposition of $I$ and $I^{sgn}$. Then, the min-sets can easily be found. In this sense, we have the following theorem that shows how a single ideal, namely $I^{ext}$, encodes all types of min-sets.

\begin{theorem} \label{thm:unified}
The primary decomposition of $I^{ext}$ encodes the unsigned min-sets and the signed min-sets. 
\end{theorem}

At first glance, the duplication of variables in $I^{ext}$ may cause Theorem \ref{thm:unified} to not result in a computational advantage. To address this, we performed simulations of computing the primary decomposition for $I^{ext}$ and $I^{sgn}$. Fig.~\ref{fig:hist_speed} shows that the speed gained by using monomial ideals overcomes the speed lost by duplicating variables. 

\begin{figure}[ht]\label{fig:hist_speed}
\centering
    \includegraphics[width=0.6\textwidth]{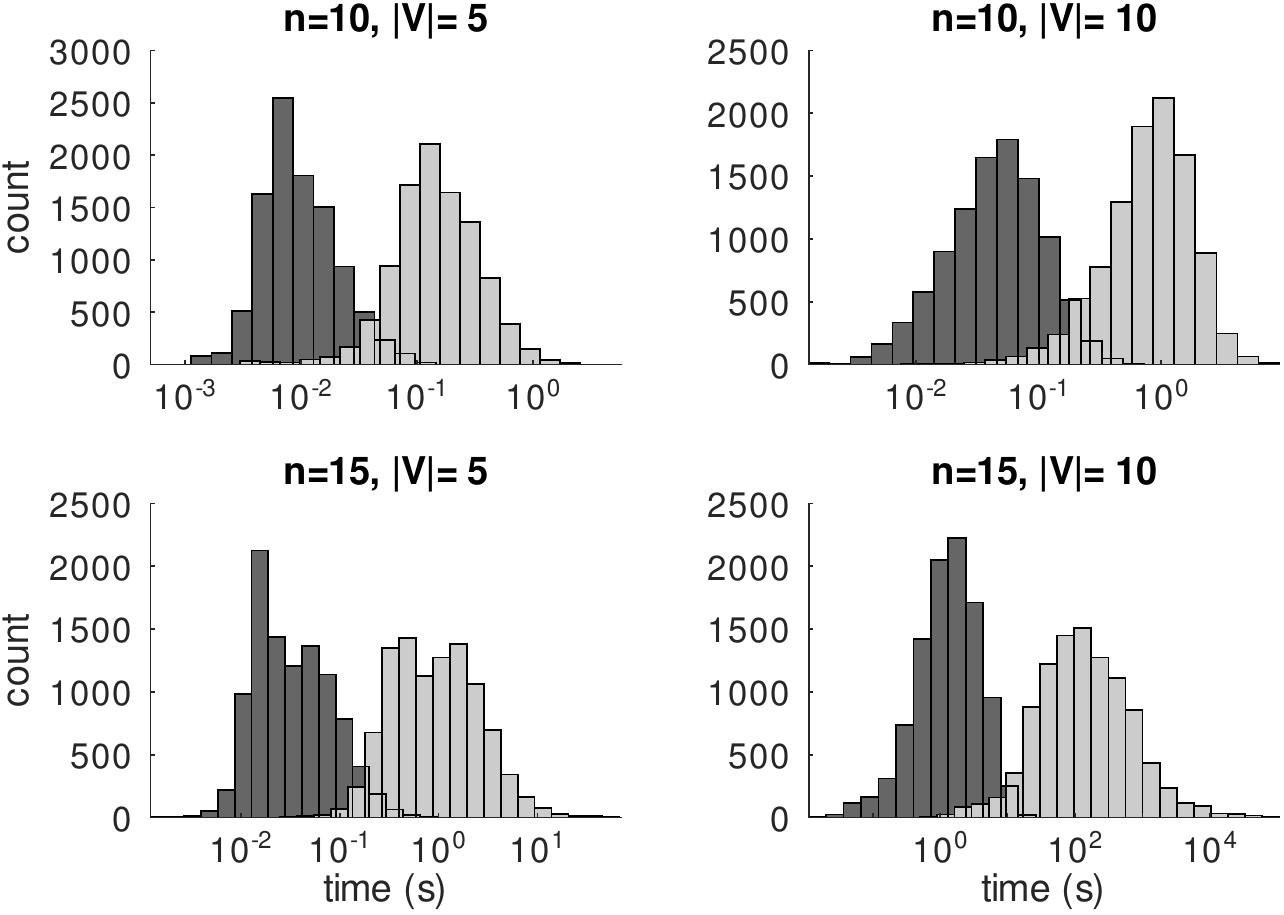}
    \caption{
Time comparison (logarithmic $x$-axis) between computing the primary decomposition of $I^{ext}$ (dark gray) vs $I^{sgn}$ (light gray). The histograms show that using the extended ideal is significantly faster than the original ideal. As the number of variables $n$, and the size of the input set $V$ increase, the difference in timing becomes more drastic.    
    }
\end{figure}

\begin{example}\label{eg:F5}
	Let $\mathcal{D}$ be the following data set with points in $\F_5^5$. 
	\begin{table}[ht]
		\centering
		\begin{tabular}{|c||c|c|c|c|c|}
			\hline
			$s_i$ & $(0,1,2,1,0)$ & $(0,1,2,1,1)$ & $(0,1,2,1,4)$ & $(3,0,0,0,0)$ & $(1,1,1,1,3)$ \\
			\hline
			$t_i$ & 0 & 0 & 1 & 3 & 4\\
			\hline
		\end{tabular}
	\caption{}
	\end{table}

We first compute the monomials
	\begin{align*}
		m^{ext}(s_1,s_3)&=m^{ext}(s_2,s_3)=x_5\\
		m^{ext}(s_1,s_5)&=m^{ext}(s_2,s_5)=x_1 \overline{x_3} x_5\\
		m^{ext}(s_2,s_4)&=m^{ext}(s_3,s_4)=x_1 \overline{x_2} \overline{x_3}\overline{x_4}\overline{x_5}\\
		m^{ext}(s_1,s_4)&=x_1\overline{x_2}\overline{x_3}\overline{x_4}\\
		m^{ext}(s_4,s_5)&=\overline{x_1}x_2 x_3 x_4 x_5\\
		m^{ext}(s_3,s_5)&=x_1 \overline{x_3} \overline{x_5}
	\end{align*}
Then, $$I^{ext} = \<x_5,x_1 \overline{x_3} x_5,  x_1 \overline{x_2} \overline{x_3}\overline{x_4}\overline{x_5}, x_1\overline{x_2}\overline{x_3}\overline{x_4}, \overline{x_1}x_2 x_3 x_4 x_5, x_1 \overline{x_3} \overline{x_5}\>$$	has the primary decomposition  
$$ 
I^{ext} = \<x_1, x_5\>\cap \<\overline{x_3}, x_5\> \cap \<\overline{x_2}, \overline{x_5}, x_5\>\cap \<\overline{x_4}, \overline{x_5}, x_5\> .
$$

Theorem \ref{thm:unified} states that this primary decomposition encodes the unsigned and signed min-sets. Indeed, if we are interested in unsigned min-sets, we simply replace $\overline{x_i}$ by $x_i$ to obtain $I$. That gives 
$$I =  
\<x_1, x_5\>\cap \<x_3, x_5\> \cap \<x_2, x_5, x_5\>\cap \<x_4, x_5, x_5\>
$$
which results in the primary decomposition 
$$I =  
\<x_1, x_5\>\cap \<x_3, x_5\> \cap \<x_2, x_5\>\cap \<x_4, x_5\>
$$
and thus the unsigned min-sets are $\{x_1,x_5\}, \{x_3,x_5\}, \{x_2,x_5\}$, and $\{x_4,x_5\}$.

If we are interested in signed min-sets, we simply replace $x_i$ by $x_i-1$ and $\overline{x_i}$ by $x_i+1$ to obtain $I^{sgn}$. That gives 
$$I^{sgn} =  
\<x_1-1, x_5-1\>\cap \<x_3+1, x_5-1\> \cap \<x_2+1, x_5+1, x_5-1\>\cap \<x_4+1, x_5+1, x_5-1\>.
$$

The last two ideals are the equal to the whole ring, so we obtain the primary decomposition 
$$I^{sgn} =  
\<x_1-1, x_5-1\>\cap \<x_3+1, x_5-1\>
$$
and thus the signed min-sets are $\{x_1,x_5\}$ and  $\{\overline{x_3},x_5\}$.

\end{example}

We remark that in general the number of primary ideals in the primary decomposition of $I^{ext}$ is a bound for the number of signed and unsigned min-sets. The next examples show that neither the number of unsigned or signed min-sets is an upper bound for the other.

\begin{example}\label{eg:nonprime_ext_nonprime_unsgn_prime_sgn}
	Let $\mathcal{D}$ be the following data set with points in $\F_2^3$. 
	\begin{table}[ht]
		\centering
		\begin{tabular}{|c||c|c|c|c|c|}
			\hline
			$s_i$ & $(0,0,0)$ & $(1,0,1)$ & $(1,1,0)$ & $(0,1,1)$ \\
			\hline
			$t_i$ & 0 & 0 & 1 & 1\\
			\hline
		\end{tabular}
	\caption{}
	\end{table}
	
	In this example the primary decomposition of $I^{ext}$ is 
	$$ I^{ext}=\<x_2\>\cap \<x_1,\overline{x_1},x_3,\overline{x_3}\>.$$
	
	Then $I=\<x_2\>\cap \<x_1,x_1,x_3,x_3\> = \<x_2\>\cap \<x_1,x_3\>$ and therefore the unsigned min-sets are $\{x_2\}$ and $\{x_1,x_3\}$. Also, 
	$I^{sgn}=\<x_2-1\>\cap \<x_1-1,x_1+1,x_3-1,x_3+1\> = \<x_2-1\>$, so there is a unique signed min-set $\{x_2\}$.
	
\end{example}

\begin{example}\label{eg:nonprime_ext_prime_unsgn_nonprime_sgn}
	Let $\D$ be the following data set  with points in $\F_3^3$. 
 
\begin{table}[ht]
	\centering
		\begin{tabular}{|c||c|c|c|c|c|}
			\hline
			$s_i$ & $(1,0,1)$ & $(0,0,0)$ & $(0,2,0)$ & $(2,1,1)$ \\
			\hline
			$t_i$ & 0 & 0 & 1 & 2\\
			\hline
		\end{tabular}
\caption{}
\end{table}
	
	In this example the primary decomposition of $I^{ext}$ is 
	$$ I^{ext}=\<x_2,\overline{x_2}\>\cap \<x_1,x_2\>\cap \<x_2,x_3\>.$$
	
	Then $I=\<x_2,x_2\>\cap \<x_1,x_2\>\cap \<x_2,x_3\> = \<x_2\>$ and therefore the unique unsigned min-set is $\{x_2\}$. Also, 
	$I^{sgn}=\<x_2-1,x_2+1\>\cap \<x_1-1,x_2-1\>\cap \<x_2-1,x_3-1\> = \<x_1-1,x_2-1\>\cap \<x_2-1,x_3-1\>$, so the signed min-sets are $\{x_1,x_2\}$ and $\{x_2,x_3\}$.
	
\end{example}

These examples show that for a signed min-set $W$, if we denote with $W^*$ the set obtained from $W$ after dropping the signs, 
then $W^*$ always contains some unsigned min-set. For instance, in Example \ref{eg:F5}, $W=\{\overline{x_3},x_5\}$ is a signed min-set and $W^*=\{x_3,x_5\}$, which contains the unsigned min-set $\{x_3,x_5\}$. Similarly, in Example \ref{eg:nonprime_ext_prime_unsgn_nonprime_sgn}, $W=\{x_1,x_2\}$ is a signed min-set and $W^*=\{x_1,x_2\}$, which contains the unsigned min-set $\{x_2\}$.
This property is always valid.

\begin{proposition}
Let $W$ be a signed min-set for a data set $\mathcal{D}$ and $W^*$ be the set obtained from $W$ by dropping the signs. Then, $W^*$ contains at least one unsigned min-set.
\end{proposition}
\begin{proof}
Suppose $W$ is a signed min-set. Then there is a unate function $f\in Mod^{sgn}(\mathcal{D})$ with $W$ as its wiring diagram. Since $f\in Mod(\mathcal{D})$ as well, $W^*$ is the unsigned support of $f$. Then, $W^*$ must contain contain some unsigned min-set. 
\end{proof}

\subsection{Sufficient condition for unique signed and unsigned min-sets }

Let $V$ be a set of inputs with $|V|=r$ and let $\Omega= \{x_1, \ldots, x_n, $ $\overline{x_1}, \ldots, \overline{x_n}\}$. We seek to construct a multiset of all possible monomials, keeping track of the signs in addition to the differences between inputs. Let $\M_\Omega$ denote this multiset. In order to compute signed min-sets, we order the data so that the outputs were non-decreasing. Due to this, when constructing these monomials, the order of inputs in $V$ is crucial and influences the sign associated to a variable. \par 
To construct this multiset, we proceed as if all outputs are different. For any two inputs, $s_i, s_j \in V$, the corresponding outputs $t_i$ and $t_j$ are such that either $t_i < t_j$ and so $s_i$ appears before $s_j$, or $t_j < t_i$ in which case $s_j$ would appear before $s_i$. This change in order only affects the sign associated with the variables in the monomial, not the variables themselves. For this reason, in constructing all possible monomials that keep track of the signs, it suffices to select an arbitrary ordering to initially form the monomials. Given two inputs, other orderings of the inputs produce either the initial monomial or its conjugate. \par 
Now that the multiset $\M_\Omega$ has been constructed, we can start answering the question of which sets of inputs correspond to unique signed minimal wiring diagrams.\par
Consider the following cases for $\M_\Omega$. We will view the generators of $I^{ext}$ in terms of $\Omega$, so that $I^{ext}$ is a squarefree monomial ideal. 
\begin{itemize}
    \item Type 1: All of the monomials in $\M_\Omega$ are univariate.
    \item Type 2: There exists a multivariate monomial $m \in \M_\Omega$ such that $\textrm{supp}(m) \cap \M_\Omega = \emptyset$.
    \item Type 3a: For every multivariate monomial $m \in \M_\Omega$, $\textrm{supp}(m) \cap \M_\Omega  \not= \emptyset$ and for every output assignment $T$, $I^{ext}$ is prime.
    \item Type 3b: For every multivariate monomial $m \in \M_\Omega$, $\textrm{supp}(m) \cap \M_\Omega  \not= \emptyset$ and there exists an output assignment $T$ such that $I^{ext}$ is not prime.
\end{itemize}

Not all data sets have signed min-sets. In particular, if both a variable and its conjugate appear, then it is not possible to find a signed min-set, which means that the signed model space is empty. 

\begin{proposition}\label{var-and-its-conj-min-set}
Let $\D$ be a set of input-output data and $I^{sgn}=\langle p_1,\ldots,p_s\rangle$ be the corresponding ideal of $R=\F[x_1,\ldots,x_n]$ generated by pseudomonomials. If, for some $k \in \{1,\ldots,n\}$, both $x_k -1$ and $x_k+1$ are in $I^{sgn}$, then $\D$ does not come from a unate function.
\end{proposition}

\begin{proof}
Suppose $I^{sgn}$ is a pseudomonomial ideal of $R=\F_3[x_1, \ldots, x_n]$  such that both $x_{i_1} +1$  and $x_{i_1}-1$ are elements of $I^{sgn}$. Since $I^{sgn}$ is an ideal, this implies that $(x_{i_1} +1)-(x_{i_1} -1)$ is also an element of $I^{sgn}$. Note that
$$(x_{i_1} +1)-(x_{i_1} -1)= x_{i_1} - x_{i_1} +1 +1 = 1 \in \F_3.$$ Thus, since $1 \in I^{sgn}$, it follows that $I^{sgn}=R$. Since $I^{sgn}$ is not a proper ideal of $R$, we are unable to find the primary decomposition of $I^{sgn}$. Since there are no signed min-sets, $Mod^{sgn}(\mathcal{D})$ must be empty and thus, the data set does not come from a unate function.
\end{proof}

\begin{example}\label{eg:no_signed_wd}
The following data set over $\F_3$ is an example of an output assignment for which there is no signed min-set.
	\begin{table}[ht]
		\centering
		\begin{tabular}{|c||c|c|c|c|}
			\hline
			$s_i$ & $(1,1,0)$ & $(1,2,0)$ & $(1,2,2)$ & $(1,0,0)$ \\
			\hline
			$t_i$ & 0 & 1 & 1 & 2 \\
			\hline
		\end{tabular}
	\caption{}
	\end{table}
	
	We have the following pseudomonomials: $p(s_1, s_2) = x_2-1$, $p(s_1, s_3)=(x_2-1)(x_3-1)$, $p(s_1, s_4)=x_2+1$, $p(s_2, s_4)=x_2+1$, $p(s_3, s_4)=(x_2+1)(x_3+1)$ and we find the find primary decomposition of the ideal $I^{sgn}=\<x_2-1, (x_2-1)(x_3-1),x_2+1, (x_2+1)(x_3+1)\>$.

	As noted in Proposition \ref{var-and-its-conj-min-set}, since both $x_2-1$ and $x_2+1$ are generators of $I^{sgn}$, $I^{sgn}$ is not a proper ideal of $R$ and so it does not have a primary decomposition, leading us to conclude that there are no signed min-sets for this data set and that the data does not come from a unate function. 
\end{example}

\begin{theorem}\label{thm:uniqueness}
If the set of inputs $V$ corresponds to $\M_\Omega$ of Type 1 or Type 3a, then $V$ has at most one signed min-set for every output assignment $T$ .

\end{theorem}

\begin{proof}
    If $\M_\Omega$ is of Type 3a, then for every output assignment $T$, $I^{ext}$ is prime. It is possible that both a variable and its conjugate are in $I^{ext}$. If such a situation occurs, then by Proposition~\ref{var-and-its-conj-min-set} no signed min-set exists. If not, then $V$ has a unique signed min-set. \par 
	Suppose $\M_\Omega$ is of Type 1; that is, suppose every monomial is univariate. We will show that for any output assignment, $I^{ext}$ is prime. Let $\Omega=\{x_1, \ldots, x_n,\overline{x_1}, \ldots, \overline{x_n}\}$. Given an arbitrary output assignment $T$, we have a submultiset $S$ of $\M_\Omega$. So, the ideal $I^{ext}$ is generated by elements of $S$, which are all univariate monomials. Since $I^{ext}$ is generated by a subset of $\Omega$, by Theorem~\ref{tfae} we conclude that it is a prime ideal. If $T$ is not an output assignment such that both a variable and its conjugate are in $I^{ext}$, then $V$ corresponds to a unique signed min-set. Therefore, if $\M_\Omega$ is of Type 1 or Type 3a, then the input set $V$ has at most one signed min-set for every output assignment. \par
	
\end{proof}

Notice that $I^{ext}$ not being a prime ideal does not guarantee that $V$ has multiple signed min-sets. For example, in Example \ref{eg:nonprime_ext_nonprime_unsgn_prime_sgn} we saw that $\M_\Omega$ was of Type 2 and indeed $I^{ext}=\langle y \rangle \cap \langle x_1, x_3, \overline{x_1}, \overline{x_3}\rangle$ is not prime; however, there is just one signed min-set, $\{ \overline{x_2} \}$, since $I^{sgn} = \langle x_2-1\rangle \cap \langle x_1-1, x_1+1, x_3-1, x_3+1\rangle = \langle x_2-1 \rangle$.

\begin{example}
	Let $V=\{s_1=(1,0,0), s_2=(1,1,0), s_3=(1,2,0), s_4=(1,2,2)\}$ be a set of inputs. Then, based on this initial ordering, we compute all possible pseudomonomials, keeping track of which inputs produce a given monomial and rewriting the pseudomonomials in terms of the expanded set $\Omega$: $\M_\Omega = \{p(s_1, s_2)= x_2-1=x_2, p(s_1, s_3)=x_2-1 = x_2, p(s_1, s_4)= (x_2-1)(x_3-1)=x_2 x_3, p(s_2, s_3)=x_2-1 = x_2, p(s_2, s_4)=(x_2-1)(x_3-1)=x_2 x_3, p(s_3, s_4) = x_3-1 =x_3, p(s_2, s_1) = x_2+1 = \overline{x_2}, p(s_3, s_1) = x_2+1 = \overline{x_2}, p(s_4, s_1) = (x_2+1)(x_3+1)=\overline{x_2 x_3}, p(s_3, s_2)= x_2+1 = \overline{x_2}, p(s_4, s_2)= (x_2+1)(x_3+1)= \overline{x_2 x_3}, p(s_3, s_2) = x_3+1=\overline{x_3}\}$. \par
	Note that $\M_\Omega$ is of Type 3. In order to establish if $\M_\Omega$ is of Type 3a or Type 3b, we must determine if it is possible to have an output assignment such that $m^{ext}$ is a multivariate generator of $I^{ext}$ and no univariate monomial that divides $m^{ext}$ is a generator of $I^{ext}$. We consider the multivariate monomials and their univariate divisors in $\M_\Omega$. The multivariate monomials are $p(s_1, s_4)= x_2x_3$,  $p(s_2, s_4)=x_2x_3$, $p(s_4, s_1)= \overline{x_2x_3}$, and $p(s_4, s_2)= \overline{x_2x_3}$. We can form the following systems for $p(s_1, s_4)$ and $p(s_2, s_4)$, respectively. Since the order of the inputs affects the corresponding pseudomonomials, we can also form systems for $p(s_4, s_1)$ and $p(s_4, s_2)$.
	
	\begin{align*}
		t_1 &< t_4 & t_2 &< t_4 & t_4 &< t_1 & t_4 &< t_2 \\
		t_1 &= t_2  &  t_1 &= t_2 & t_1 &= t_2  &  t_1 &= t_2\\
		t_1 &= t_3 & t_1 &= t_3 & t_1 &= t_3 & t_1 &= t_3 \\
		t_2 &= t_3 & t_2 &= t_3 & t_2 &= t_3 & t_2 &= t_3\\
		t_3 &= t_4 &  t_3 &= t_4 & t_3 &= t_4 &  t_3 &= t_4
	\end{align*}

	All of the systems are inconsistent. This implies that, for every output assignment, $I^{ext}$ is a prime ideal. Thus, $\M_\Omega$ is of Type 3a and so, by Theorem \ref{thm:uniqueness}, $V$ has at most one signed min-set for every output assignment. 
	
\end{example}

The following corollary specializes Theorem~\ref{thm:uniqueness} for unsigned min-sets, where the original ideal is already square-free and $\Omega$ is the original sets of variables without extending them.
\begin{corollary}
    Let $V$ be a set of inputs corresponding to $\mathcal{M}$ of Type 1 or 3a. Then $V$ has a unique unsigned min-set.
\end{corollary}

\section{A combinatorial  approach to uniqueness}

In this section we study the problem of uniqueness from a combinatorial point of view. Namely, we present results of the relationship between uniqueness and the way the data is distributed in the hypercube $\F^n$.

\subsection{Necessary conditions for uniqueness}

We say that $V$ has a \textit{diagonal} if there is a point $p$ such that it differs from all other points  in $V$ in at least two entries. If $l$ is the maximum number such that $p$ differs from all other points in $V$ in at least $l$ entries, then $l$ is called the \textit{length} of the diagonal. Fig.~\ref{fig:diag} shows sets with a diagonal.

\begin{figure}[ht]\label{fig:diag}
\centering
    \includegraphics[width=0.4\textwidth]{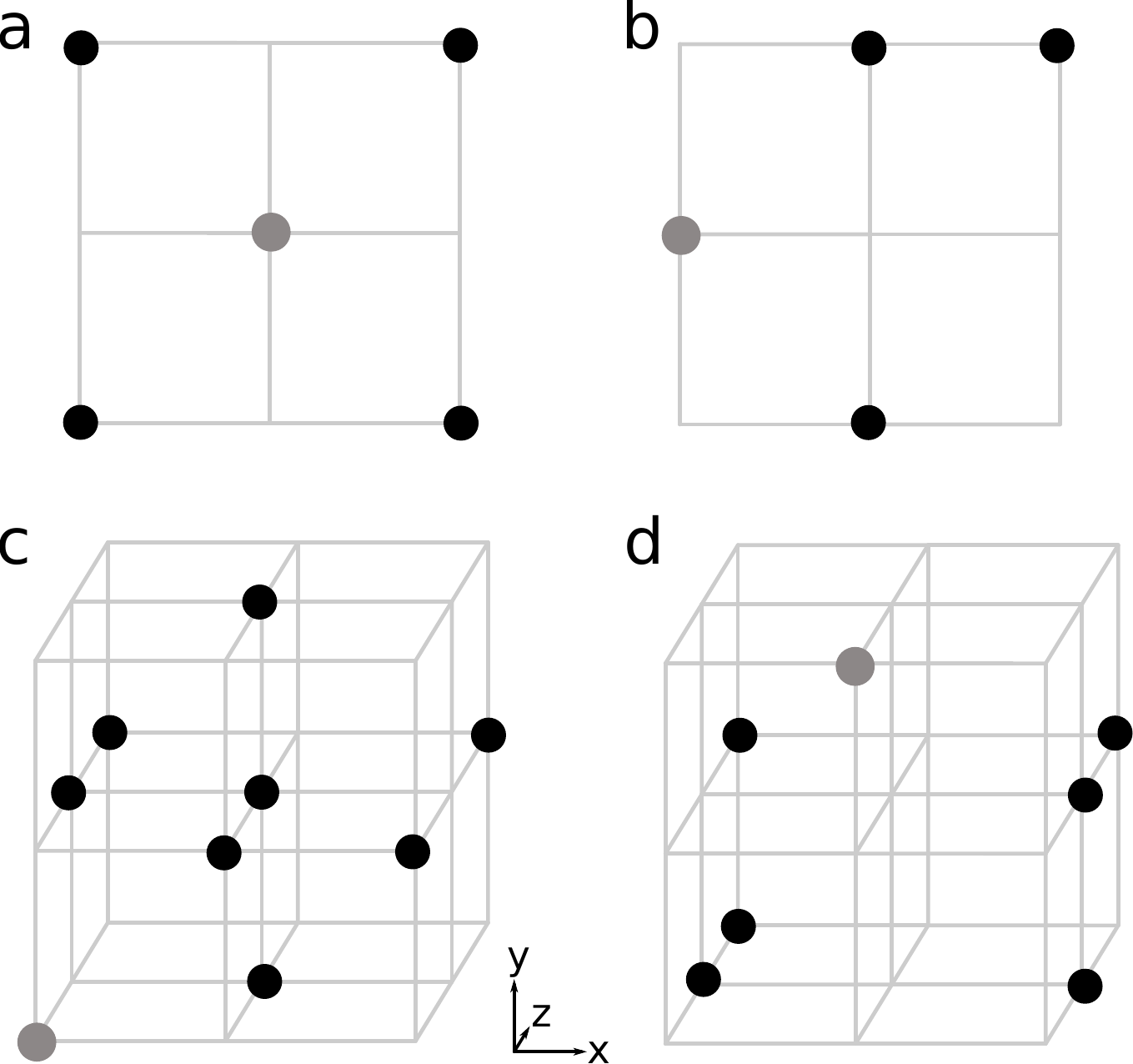}
    \caption{
    Examples of sets with diagonals. The gray points indicate a point that differs from all others in the largest number of entries. The length of the diagonal is 2 for (a), (b), (c), and it is 3 for (d). 
    }
\end{figure}

\begin{theorem}\label{thm:diag_unsigned}
If $V$ has a diagonal of length $l$, then there is an output assignment with at least $l$ unsigned min-sets.
\end{theorem}

\begin{proof}
Let $p$ be the point that differs from all other points in at least two entries. Without loss of generality, suppose $p=\textbf{0}=(0,\ldots, 0)$. This is achieved by using bijective functions for each variable that map the nonzero entries of $p$ to 0 (each entry may have its own function).
 We consider the assignment $\textbf{0}\rightarrow 0$ and $s\rightarrow 1$ for all other $s\in V$.

Each $s\in V\setminus\{\textbf{0}\}$ differs from $\textbf{0}$ by at least $l$ entries, so $m(s,\textbf{0})$ is always a multivariate monomial with at least $l$ factors. Then, $I$ is generated only by multivariate monomials with at least $l$ factors and hence will have at least $l$ primary components. Therefore, there are at least $l$ unsigned min-sets.
\end{proof}

Theorem \ref{thm:diag_unsigned} guarantees that there are output assignments for the input sets in Fig.~\ref{fig:diag} that result in more than one unsigned min-set. This theorem is not valid in the signed case, as the next example shows.

\begin{example}\label{eg:diag}
Consider the input set $V=\{(0,0),(2,0),(0,2),(2,2),(1,1)\}\subset \F_3^2$ illustrated in Fig.~\ref{fig:diag}a. It has a diagonal of length 2 since the point $(1,1)$ differs from all others in two entries. Also, $(1,1)$ is the only such point. Since $V$ has a diagonal, by Theorem \ref{thm:diag_unsigned} there is an output assignment that results in multiple unsigned min-sets. However, by exhaustive analysis it can be shown that any output assignment results in at most one signed min-set.
\end{example}

If we restrict it to Boolean data, however, Theorem~\ref{thm:diag_unsigned} is valid in the signed case and stated as the next result.

\begin{theorem}
If $V\subset \{0,1\}^n$ has a diagonal of length $l$, then there is an ouput assignment with at least $l$ signed min-sets.
\end{theorem}

\begin{proof}
By changing 0/1 to 1/0 as needed, without loss of generality we suppose $\textbf{0}=(0,\ldots, 0)$ is the point that differs from all other points in at least $l$ entries. Note that this change can only affect the sign of the min-sets, but not how many there are. We consider the assignment $\textbf{0}\rightarrow 0$ and $s\rightarrow 1$ for all other $s\in V$.

Since each $s\in V\setminus\{\textbf{0}\}$ differs from $\textbf{0}$ by at least $l$ entries, $m^{sgn}(s,\textbf{0})$ has at least $l$ factors of the form $x_i-1$ and no factor of the form $x_i+1$. Then, $I^{sgn}$ is generated only by polynomials of the form $\prod_{i\in M}  (x_i-1)$, where $|M|\geq l$. If the ideal $I^{sgn}$ had less than $l$ primary components, then it would  follow that the ideal $I^*$ obtained by making the replacement $x_i-1\rightarrow x_i$ would also have less than $l$ primary components. But $I^*$ is generated by multivariate monomials that have at least $l$ factors. This contradiction implies that $I^{sgn}$ has at least $l$ primary components and hence there are at least $l$ signed min-sets.
\end{proof}

There is a generalization of the previous theorem which is presented next, but it needs a stronger hypothesis than just having a diagonal.

\begin{theorem}\label{thm:diag_signed}
If $V$ has a diagonal of length $l$ and the corresponding point $p$ is a corner point, then there is an output assignment with at least $l$ signed min-sets.
\end{theorem}

\begin{proof}
Without loss of generality, suppose $\textbf{0}=(0,\ldots ,0)$ is the point that differs from all other points in at least 2 entries. This is achieved by using the decreasing function $\neg(z):=M-z$ to make change of variables as needed, where $M$ is the maximum value variables can take. Note that $\neg$ being monotone is key as it does not change the number of min-sets, but just their signs.

We consider the assignment $\textbf{0}\rightarrow 0$ and $s\rightarrow 1$ for all other $s\in V$. The rest of the proof is the same as the previous theorem.
\end{proof}

Example~\ref{eg:diag} shows that the condition of $p$ being a corner point cannot be omitted.

\begin{example}
Consider the input set $V=\{(0,0,0),(1,1,0),(0,1,1),(1,0,1)\}\subset \{0,1\}^3$ with outputs $T=\{0,0,0,1\}$. The length of the diagonal in $V$ is 2, and $I^{ext}=\< x_1 x_3,\overline{x_2} x_3,x_1 \overline{x_2} \>$ has primary decomposition $I^{ext}=\<x_1,\overline{x_2}\>\cap \<\overline{x_2},x_3\>\cap \<x_1,x_3\>$. 
Then, we have the following primary decompositions $I=\<x_1,x_2\>\cap \<x_2,x_3\>\cap \<x_1,x_3\>$ and $I^{sgn}=\<x_1-1,x_2+1\>\cap \<x_2+1,x_3-1\>\cap \<x_1-1,x_3-1\>$. So we have three unsigned and three signed min-sets. That is, in Theorems \ref{thm:diag_unsigned}-\ref{thm:diag_signed}, it is possible to have more min-sets than the length of the diagonal.
\end{example}
	
\subsection{Necessary and sufficient conditions for uniqueness}

We now state the definitions that will be needed for necessary and sufficient conditions for uniqueness of unsigned and signed min-sets.

\begin{definition}
    We say $C$ is a cylinder if $C=\{s:s_i=u_i \text{ for } i\in N\}$ for some $(u_i)_{i\in N}$ and $N\subset \{1,\ldots,n\}$. Such sets $C$ can be constructed using two  points $p$, $q$, and defining the cylinder by $\mathcal{C}(p,q)=\{s: s_i=p_i=q_i\text { for all $i$ such that $p_i=q_i$} \}$. That is, $\mathcal{C}(p,q)$ is the set of points where only the entries where $p$ and $q$ differ are allowed to vary. 
\end{definition}
Note that $p,q\in \mathcal{C}(p,q)$. Fig.~\ref{fig:setsH} illustrates some cylinders in $\{0,1,2\}^3$. 

\begin{figure}[ht]\label{fig:setsH}
\centering
    \includegraphics[width=0.5\textwidth]{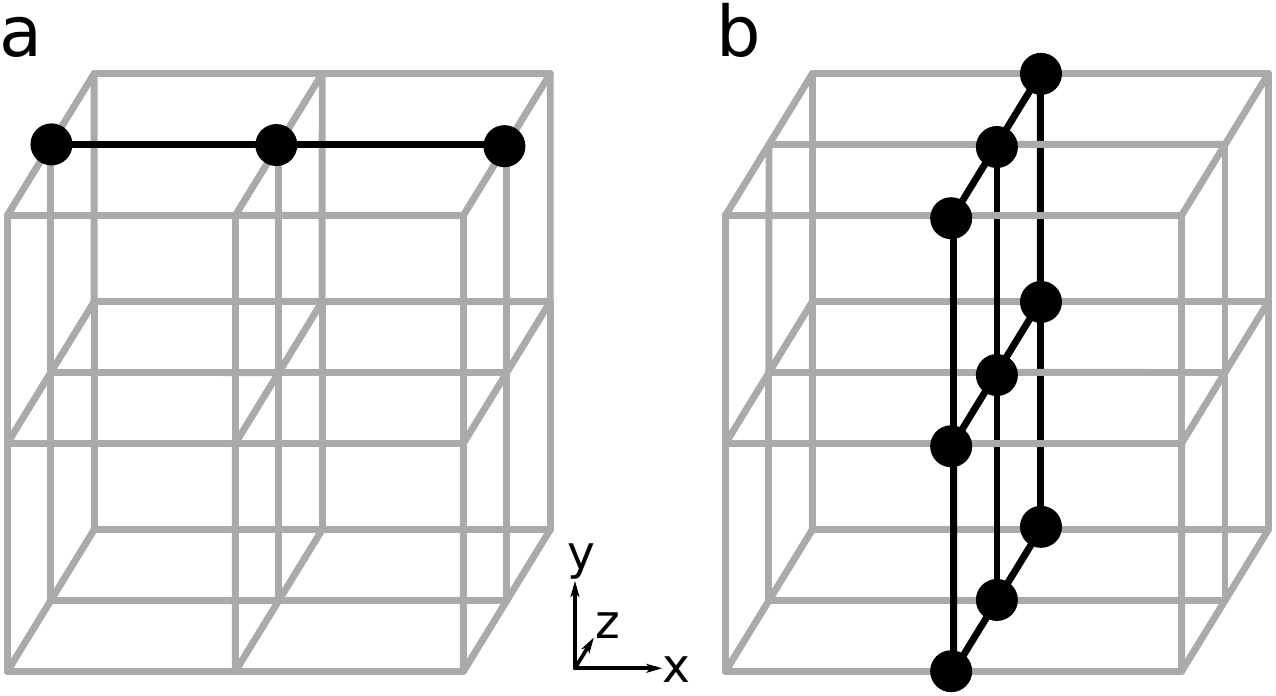}
    \caption{
    Examples of cylinders in $\{0,1,2\}^3$ (the inset with arrows shows the orientation of the $xyz$ space). (a) The cylinder $C=\{s:s_2=2,s_3=1\}$ (indicated by the points)  can also be seen as $\mathcal{C}(021,221)$, $\mathcal{C}(121,221)$, and $\mathcal{C}(021,121)$. (b)The cylinder $C=\{s:s_1=1\}$  can also be seen as $\mathcal{C}(100,122)$, $\mathcal{C}(120,111)$, among others.   Note that $\{0,1,2\}^3$ is also a cylinder. (Parentheses and commas are omitted in listing points in $\{0,1,2\}^3$ for ease of reading.)
    }
\end{figure}

\begin{definition}[Cylindrically Connected]
We say that a set is \emph{connected} if for every $p$ and $q$ in the set, there is a sequence $p=s_0,s_1,s_2,\ldots,s_{l-1},s_l=q$ in the set such that $d(p,s_1)=d(s_1,s_2)=\ldots=d(s_{l-1},q)=1$. A set with a single element is defined as connected.  
If $C\cap V$ is connected for any cylinder $C$ we say that $V$ is \emph{cylindrically connected}.
\end{definition}

\begin{example}
The set $V\subset \{0,1,2,3\}^2$ shown in Fig.~\ref{fig:con_disc_setsH_2d}a is connected, but the set in Fig.~\ref{fig:con_disc_setsH_2d}b is not. 
The subset of $\{0,1,2\}^3$ shown in Fig.~\ref{fig:con_disc_setsH_3d} is cylindrically connected and the subset shown in Fig.~\ref{fig:notcon_disc_setsH_3d} is not.
\end{example}

\begin{figure}[ht]\label{fig:con_disc_setsH_2d}
\centering
    \includegraphics[width=0.5\textwidth]{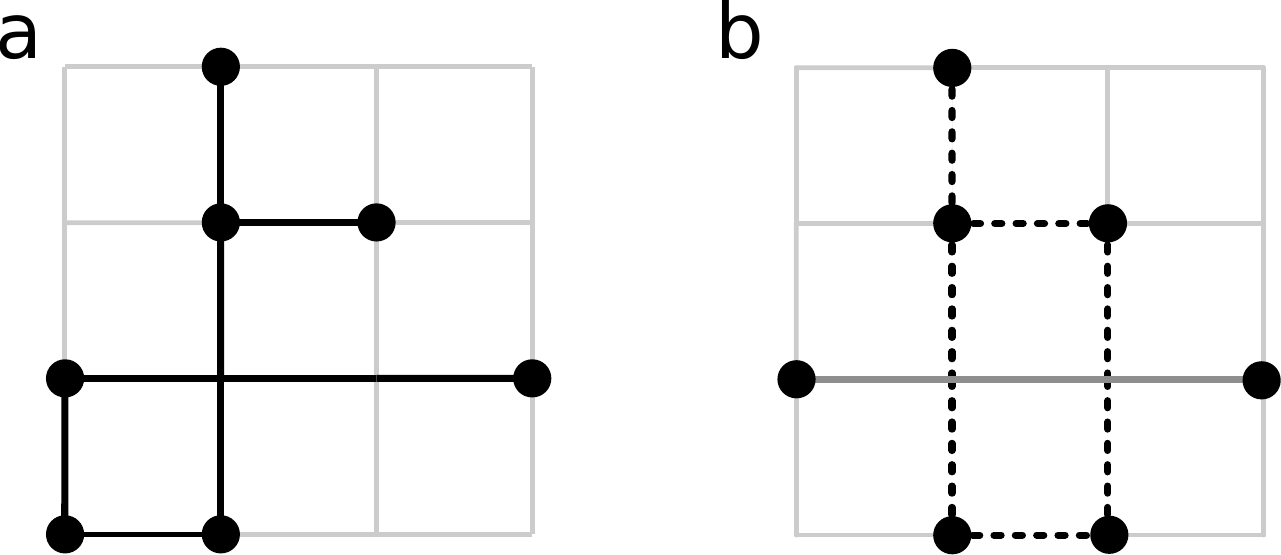}
    \caption{ (a) A connected set. Any point can be reached from another by a sequence such that consecutive points have distance 1. (b) A disconnected set. Some points cannot be reached from others. Points connected by dashed line can be reached from each other. Points connected by black line can also be reached from each other. However, points from the two different groups cannot be reached from each other.
    }
\end{figure}

\begin{figure}[ht]\label{fig:con_disc_setsH_3d}
\centering
    \includegraphics[width=0.5\textwidth]{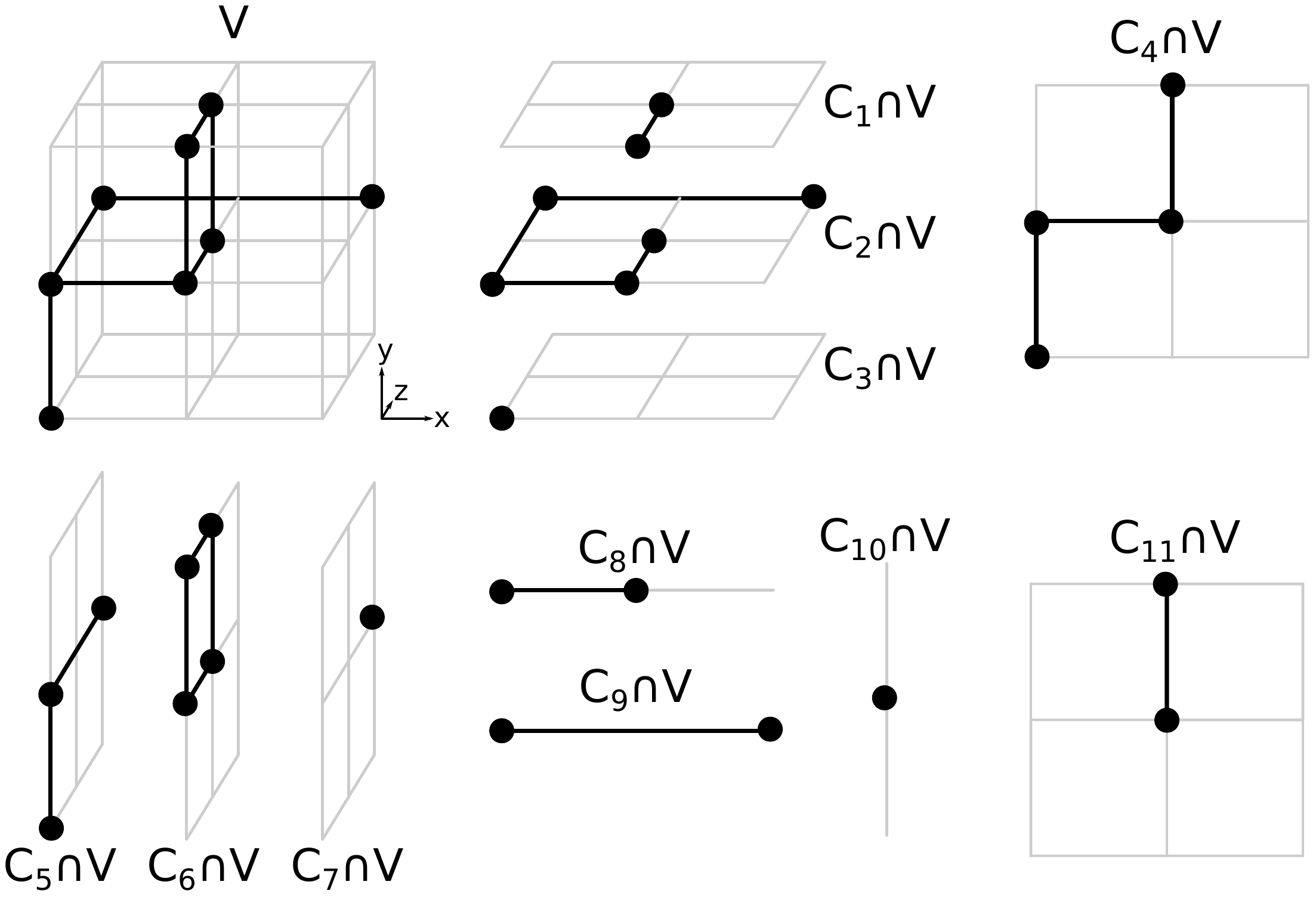}
    \caption{Example of an input set and some intersections of the form $C\cap V$. The black lines show that each of those intersections are connected. $C_1=\{s:s_2=2\}$,
    $C_2=\{s:s_2=1\}$,
    $C_3=\{s:s_2=0\}$,
    $C_4=\{s:s_3=0\}$
    $C_5=\{s:s_1=0\}$,
    $C_6=\{s:s_1=1\}$,
    $C_7=\{s:s_1=2\}$,
    $C_8=\{s:s_2=1,s_3=0\}$,
    $C_9=\{s:s_2=1,s_3=2\}$,
    $C_{10}=\{s:s_1=2,s_3=2\}$,
    $C_{11}=\{s:s_3=1\}$. For all other cylinders $C$, $C\cap V$ is connected.
    }
\end{figure}

\begin{theorem}\label{thm:uniqueness_unsigned_Boolean}
Consider $V$ to be an input set. The following are equivalent.

\begin{enumerate}
    \item For every output assignment, there is exactly one unsigned min-set.
    
    \item For every cylinder $C$, if $S  \subsetneq C\cap V$ is connected, then there is a connected set $S'$ such that $S\subsetneq S'\subset C\cap V$.
    
    \item $V$ is cylindrically connected.

\end{enumerate}
\end{theorem}

\begin{proof}
(1 $\implies $ 2)
Let $C=\{s:s_i=u_i \text{ for } i\in N\}$ be a cylinder and suppose $S\subsetneq C\cap V$ is connected and let $q\in C\cap V, q\notin S$. Consider the following output assignment: $s\rightarrow 0$ for $s\in S$ and $s\rightarrow 1$ for all other $s\in V$. Note that the output assignment for  $q$ is $1$ (since $q\notin S$). Also consider $p\in S$; it will have an output assignment of 0. Then $m(p,q)$ is a monomial in $I$.

Since there is a unique min-set, $I$ is primary and $m(p,q)$ must have a univariate divisor, say, $m(s,r)$ with output assignments $s\rightarrow 0$ and $r\rightarrow 1$, where $s\in S$, and $r\in V\setminus S$.

Note that $r\in C\cap V$. Indeed, since $s,p,q\in C$, $s_i=p_i=q_i$ for $i\in N$, then, for $i\in N$, $m(p,q)$ does not have any factor $x_i$. Since  $m(s,r)$ divides $m(p,q)$, $m(s,r)$ cannot be $x_i$ for $i\in N$, so $r_i=s_i=p_i=q_i$ for $i\in N$. This means that $r\in C$. Then, $S':=S\cup \{r\}$ satisfies $S\subsetneq S'\subset C\cap V$. 

It remains to show that $S'=S\cup \{r\}$ is connected. This follows from the fact that $S$ is connected, $s\in S$, and $d(s,r)=1$ (since $m(s,r)$ is univariate). 

(2 $\implies$ 3) 
Suppose $p,q\in C\cap V$. If $d(p,q)\leq 1$, then the proof follows. If $d(p,q)\geq 2$, starting with the connected set $S=\{p\}$, one can use (2) inductively to construct a bigger connected set $S'$ that eventually becomes all of $C\cap V$. So $C\cap V$ will be connected. Note that this is the same proof used in Theorem \ref{thm:uniqueness_unsigned_Boolean} for this part of the theorem.

(3 $\implies$ 1)
Consider a fixed output assignment. We will prove that every multivariate monomial in $I$ has a univariate divisor in $I$. This will imply that $I$ is primary and hence there is a unique min-set.

Suppose $m(p,q)\in I$ has more than two factors. Then, $p,q\in V$ and the output assignments for $p$ and $q$ are different. 

Since $\mathcal{C}(p,q)\cap V$ is connected, there is a sequence $p=s_0,s_1,\ldots,s_l=q\in \mathcal{C}(p,q)\cap V$ such that $d(p,s_1)=d(s_1,s_2)=\ldots=d(s_{l-1},q)=1$. Since the outputs of $p$ and $q$ are different, there are two consecutive elements of the sequence such that their outputs are different, say $s_k$ and $s_{k+1}$. Then, $m(s_k,s_{k+1})$ is a monomial (of degree 1, since $d(s_k,s_{k+1})=1$) in $I$ that divides $m(p,q)$ (since $s_k,s_{k+1}\in \mathcal{C}(p,q)$).

\end{proof}

\begin{figure}[ht]\label{fig:notcon_disc_setsH_3d}
\centering
    \includegraphics[width=0.4\textwidth]{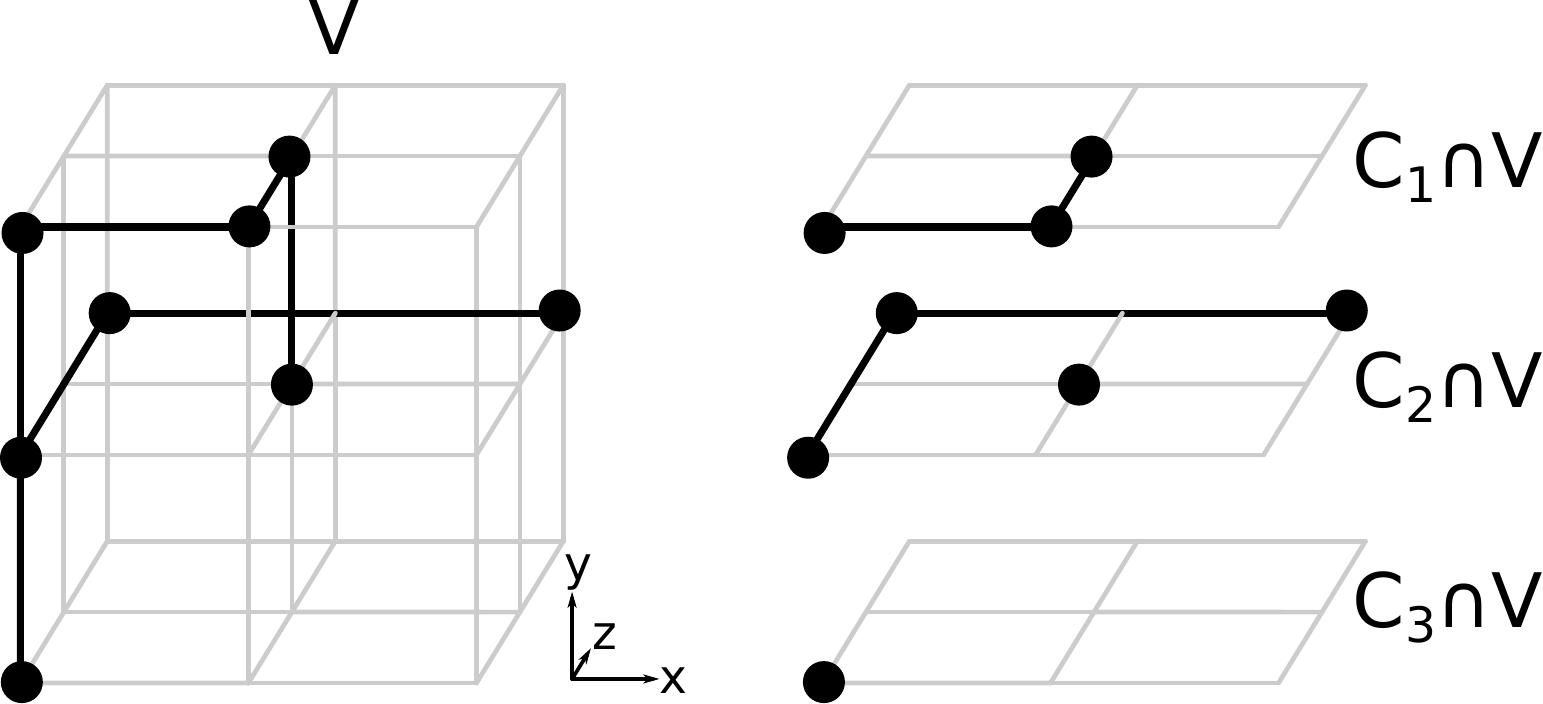}
    \caption{Example of an input set that is not cylindrically connected and some intersections of the form $C\cap V$. The black lines highlight which points are connected.  $C_1=\{s:s_2=2\}$,
    $C_2=\{s:s_2=1\}$,
    $C_3=\{s:s_2=0\}$. Note that $C_2\cap V$ is not connected, since the middle point cannot be reached from the others. Then, $V$ is not cylindrically connected.
    }
\end{figure}

\begin{example}\label{eg:1not2}
Consider the input set $V=\{(0,0),(2,0),(0,2),(2,2),(1,1)\}\subset\F_3^2$. $V$ is not connected, but  by exhaustive analysis it can be shown that any output assignment results in at most one signed min-set. Thus, (1) does not imply (2) in Theorem \ref{thm:uniqueness_unsigned_Boolean} if we consider signed min-sets.
\end{example}

\begin{example}\label{eg:2not1}
Consider the input set $V=\{(0,0,0),(0,2,0),(2,2,0),(2,2,1),(2,1,1)\}\subset\F_3^3$ and output assignment $(2,1,1)\rightarrow 0$ and $s\rightarrow 1$ for all other $s\in V$. $V$ satisfies condition (2) in Theorem \ref{thm:uniqueness_unsigned_Boolean}, but there are two signed min-sets. Thus, (2) does not imply (1) in Theorem \ref{thm:uniqueness_unsigned_Boolean} if we consider signed min-sets.
\end{example}

If the data is Boolean, however, Theorem \ref{thm:uniqueness_signed_Boolean} is valid for signed min-sets as well. 

Before we continue to signed min-sets, we note that there are other results which guarantee a unique unsigned min-set. For example, in~\cite{dimitrova-arxiv-22}, we showed that if the vanishing ideal $I(V)$ has a unique reduced Gr\"obner basis, then $V$ will correspond to a unique unsigned min-set for any output assignment; and in earlier work~\cite{he-unique-gbs} we provided a sufficient condition on $V$ for $I(V)$ to have a unique reduced Gr\"obner basis.

\begin{theorem}\label{thm:uniqueness_signed_Boolean}
Consider $V$ to be a Boolean input set. The following are equivalent.

\begin{enumerate}
    \item For every output assignment, there is at most one signed min-set.
    
    \item For every cylinder $C$, if $S  \subsetneq C\cap V$ is connected, then there is a connected set $S'$ such that $S\subsetneq S'\subset C\cap V$.
    
    \item $V$ is cylindrically connected. 
\end{enumerate}
If any of these is true and the data comes from a unate function, then there exists a unique signed min-set.

\end{theorem}

\begin{proof}

(1 $\implies $ 2)
Let $C=\{s:s_i=u_i \text{ for } i\in N\}$ be a cylinder and suppose $S\subsetneq C\cap V$ is connected and let $q\in C\cap V, q\notin S$. Consider $p\in S$ such that $d(p,q)=\min\{d(s,q),s\in S\}$, the closest point to $q$ in $S$. Consider the following output assignment: $p\rightarrow 0$ and $s\rightarrow 1$ for all other $s\in V$. Note that the output assignment for  $q$ is $1$. Now, by making the change of variables 0/1 to 1/0 as needed, we can assume that $p=\textbf{0}$. 
Note that this change of variables does not change the number of min-sets, but only their signs.

Since $I^{sgn}=\< m^{sgn}(\textbf{0},s): s\in V\setminus \{\textbf{0}\}\>$, all generators of $I^{sgn}$ only have factors of the form $x_i-1$. Because of this and since $I^{sgn}$ is prime, any generator of $I^{sgn}$ must have a univariate divisor that is also a generator. 

Since $m^{sgn}(p,q)=m^{sgn}(\textbf{0},q)=\prod_{q_i\neq 0}(x_i-1)=\prod_{q_i=1}(x_i-1)$ is a generator in $I^{sgn}$, there must be another univariate generator, $m^{sgn}(\textbf{0},r)=x_{i_0}-1$ that divides $m^{sgn}(\textbf{0},q)$. Note that this means that $r\in V$ and that $d(\textbf{0},r)=1$ and $d(r,q)=d(\textbf{0},q)-1$.

Note that $r\in C\cap V$. Indeed, since $p=\textbf{0},q\in C$, $p_i=0=q_i$ for $i\in N$, then, for $i\in N$, $m(\textbf{0},q)=\prod_{q_i=1}(x_i-1)$ does not have any factor of the form $x_i-1$. Since  $m(\textbf{0},r)$ divides $m(\textbf{0},q)$, $m(\textbf{0},r)$ cannot be $x_i-1$ for $i\in N$, so $r_i=p_i=0=q_i$ for $i\in N$. This means that $r\in C$. We also claim that $r\notin S$. This follows from the fact that $d(r,q)<d(\textbf{0},q)$ and $d(p,q)=d(\textbf{0},q)$ is minimal.

Then, $S':=S\cup \{r\}$ satisfies $S\subsetneq S'\subset C\cap V$. It remains to show that $S'=S\cup \{r\}$ is connected. This follows from the fact that $S$ is connected, $p=\textbf{0}\in S$, and $d(\textbf{0},r)=1$.

(2 $\implies$ 3) Suppose $p,q\in C\cap V$. If $d(p,q)\leq 1$, then the proof follows. If $d(p,q)\geq 2$, starting with the connected set $S=\{p\}$, one can use (2) inductively to construct a bigger connected set $S'$ that eventually becomes all of $C\cap V$. So $C\cap V$ will be connected. 

(3 $\implies$ 1)
Consider a fixed output assignment. We will prove that any multivariate generator in $I$ has a univariate divisor in $I$. That will imply that $I$ is primary and hence there is a unique min-set. 

Consider any generator of $I^{sgn}$ and let $m^{sgn}(p,q)$ be a divisor of minimal degree that is one of the generators of $I^{sgn}$. We will prove that $m^{sgn}(p,q)$ is univariate. 

First, note that $p,q\in V$ and the output assignments for $p$ and $q$ are different. Now, by making the change of variables 0/1 to 1/0 as needed, we can assume that $p=\textbf{0}$ and that the output corresponding to $p$ is 0. 
Note that this change of variables does not change the number of min-sets, but only their signs. Then, $m^{sgn}(p,q)$ becomes $m^{sgn}(\textbf{0},q)=\prod_{q_i=1}(x_i-1)$.

Since $\mathcal{C}(\textbf{0},q)\cap V$ is connected, there is a sequence $\textbf{0}=s_0,s_1,\ldots,s_l=q\in \mathcal{C}(\textbf{0},q)\cap V$ such that $d(\textbf{0},s_1)=d(s_1,s_2)=\ldots=d(s_{l-1},q)=1$. Since the outputs of $\textbf{0}$ and $q$ are different, we can pick the first element of the sequence with output not equal to 0, say $s_k$. Then, 
$m^{sgn}(\textbf{0},s_k)$ and $m^{sgn}(s_{k-1},q)$ are generators of $I^{sgn}$ that divide $m^{sgn}(\textbf{0},q)$. Since $m^{sgn}(\textbf{0},q)$ is of minimal degree, it follows that $s_{k-1}=\textbf{0}$ and $s_k=q$. Then, the sequence is just $\textbf{0}=s_0,s_1=q$ and $d(\textbf{0},q)=1$. This means that $m^{sgn}(p,q)$ is univariate. 

\end{proof}

\subsection{Design of experiments}
Here we show how our results can guide the design of experiments process when the goal is to obtain a unique wiring diagram starting from existing data. 

\begin{example}\label{eg:exp_design_cube}
Consider the unate Boolean function $f:\{0,1\}^3\rightarrow \{0,1\}$ given by $f(x)=x_1 \vee \overline{x_3}$ that will be used to generate data only. Consider the input set   $V=\{000,100,101,011\}$ and suppose we are interested in signed min-sets. In this case the data set is 
$$
\mathcal{D}=\{ (000,1), (100,1), (101,1), (011,0) \},
$$ 
illustrated in Fig.~\ref{fig:exp_design_cube}a (parentheses and commas are omitted from the elements of $\{0,1\}^3$). In this case $I^{sgn}=\<x_2+1\>\cap \<x_1-1,x_3+1\>$ and the signed  min-sets are $\{\overline{x_2}\}$ and $\{x_1,\overline{x_3}\}$. $V$ is not cylindrically connected (e.g. $V\cap\{s:s_1=0\}$ is not connected), so uniqueness of min-sets is not guaranteed. If we want to choose which additional experiment to add to get a unique signed min-set, we have four possible ways to extend $V$ shown in Fig.~\ref{fig:exp_design_cube}b. Only the last extension results in a set that is cylindrically connected, so we pick the experiment involving $001$ ( $f(001)=0$ in this example ) to obtain uniqueness regardless of what the function $f$ is (which is not known \textit{a priori} in practice). Indeed, adding the data point $(001,0)$ to $\mathcal{D}$  we obtain the unique signed min-set $\{x_1,\overline{x_3} \}$. The other extensions are not in general guaranteed to result in a unique signed min-set. With the particular example of $f$ that we have, all of these extensions result in more than one signed min-set.
\end{example}

\begin{figure}[ht]\label{fig:exp_design_cube}
\centering
    \includegraphics[width=0.3\textwidth]{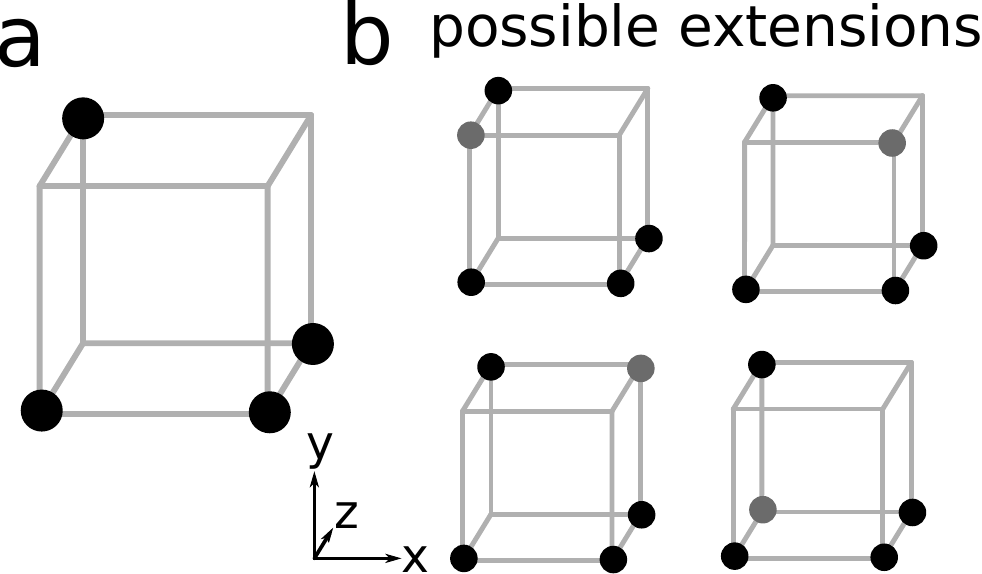}
    \caption{Input set in Example~\ref{eg:exp_design_cube}.
    }
\end{figure}

\begin{example}\label{eg:exp_design_plane}
Consider the function $f:\{0,1,2\}^2\rightarrow \{0,1,2\}$ given by $f(x)=x_1^2+x_1\mod 3$ that will be used to generate data only. Consider the input set   $V=\{00,20,12\}\subset \{0,1,2\}^2$ (parentheses and commas are omitted from the elements of $\{0,1,2\}^2$) and suppose we are interested in unsigned min-sets. In this case the data set is $\mathcal{D}=\{(00,0), (20,0), (12,2)\}$, Fig.~\ref{fig:exp_design_plane}a. Here, $I^{sgn}=\<x_1\>\cap \<x_2\>$ and the unsigned  min-sets are $\{x_1\}$ and $\{x_2\}$. $V$ is not cylindrically connected ($V\cap\{0,1,2\}^2=V$ is not connected), so uniqueness of min-sets is not guaranteed. If we want to choose which additional experiments to add to get a unique signed min-set, we have six possible ways to extend $V$, Fig.~\ref{fig:exp_design_plane}b. Only the first three extensions result in a set that is cylindrically connected, thus we pick the experiment involving $02$, $22$, or $10$ to obtain uniqueness regardless of what the function $f$ is (which is not known \textit{a priori} in practice). Those three extensions result in the unique unsigned min-set $\{x_1\}$. The last three extensions are not in general guaranteed to result in a unique unsigned min-set. With this particular example of $f$, all of these extensions result in two unsigned min-sets.
\end{example}

\begin{figure}[ht]\label{fig:exp_design_plane}
\centering
    \includegraphics[width=0.4\textwidth]{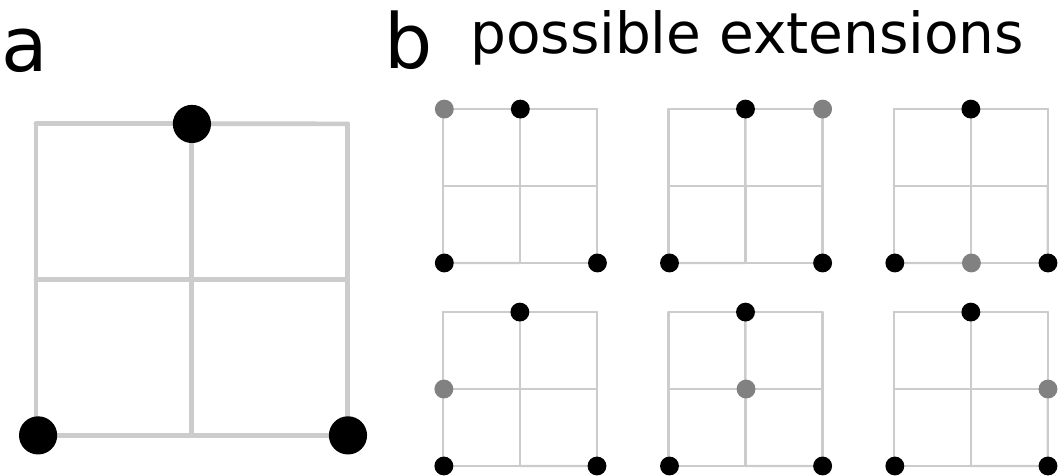}
    \caption{Input set in Example~\ref{eg:exp_design_plane}.
    }
\end{figure}

\section{Discussion/Conclusion}

Using algebraic models for gene regulatory networks and other biological systems has proven advantageous in particular since it allows to generate all minimal wiring diagrams consistent with the data. However, the number of minimal wiring diagrams (equivalently, min-sets) can be very large and a necessary and sufficient condition on the input data for the uniqueness of the min-set was unknown until now. We studied this problem from an algebraic and combinatorial points of view. The algebraic approach provides a unified framework to study signed and unsigned min-sets simultaneously. We used this to give a sufficient condition for uniqueness of signed and unsigned min-sets. The combinatorial approach resulted in necessary and sufficient conditions for uniqueness of unsigned min-sets and for Boolean signed min-sets. While uniqueness of signed min-sets for non Boolean data remains an open problem, the results introduced in this manuscript advance the study of algebraic design of experiments and enable more efficient data and model selection.

\bibliographystyle{siamplain}
\bibliography{references}

\end{document}